\documentclass[11 pt]{amsart}
\usepackage {amssymb,latexsym,amsthm,amsmath,amsfonts, mathtools, multirow, longtable, hyperref}
\usepackage{enumitem}
\usepackage{hyperref}

\long\def\symbolfootnote[#1]#2{\begingroup
	\def\thefootnote{\fnsymbol{footnote}}\footnote[#1]{#2}\endgroup}
\newcommand{\Z}{\mathbb Z}
\newcommand{\Q}{\mathbb Q}
\newcommand{\R}{\mathbb R}

\makeatletter
\def\imod#1{\allowbreak\mkern10mu({\operator@font mod}\,\,#1)}
\makeatother
\newtheorem{theorem}{Theorem}
\newtheorem{lemma}[theorem]{Lemma}
\newtheorem{corollary}[theorem]{Corollary}
\newtheorem{proposition}[theorem]{Proposition}
\newtheorem*{theorem*}{Theorem}
\theoremstyle{definition}
\newtheorem{defn}[theorem]{Definition}

\numberwithin{equation}{section}

\title[On the growth of cuspidal cohomology of ${\rm GL}_4$]{On the growth of cuspidal cohomology of ${\rm GL}_4$}
\date{\today}

\author{\bf Chandrasheel Bhagwat \ \& \ Sudipa Mondal}
\address{Indian Institute of Science Education and Research, Dr.\,Homi Bhabha Road, Pashan, Pune 411008,  INDIA.}
\email{cbhagwat@iiserpune.ac.in, \ sudipa.mondal123@gmail.com}


\begin{document}

\begin{abstract}
In this article, we establish an asymptotic estimate on the number of cuspidal automorphic representations of ${\rm GL}_4(\mathbb A_{\mathbb Q})$ which contribute to the cuspidal cohomology of ${\rm GL}_4$ and are obtained from symmetric cube transfer of automorphic representations of  ${\rm GL}_2(\mathbb A_{\mathbb Q})$ of a given weight and with varying level structure. This generalises the recent work of C. Ambi about the similar problem for ${\rm GL}_3$.
\end{abstract}

\maketitle
	
\section{Introduction}
	Let $d$ be a positive square free integer and $\mathbb{E}=\Q(\sqrt{-d})$ be an imaginary quadratic extension of $\Q$. Let $\chi$ be a gr\"o{\ss}encharacter of the group of id\`eles over the field $\mathbb{E}$. Using Langlands functoriality and automorphic induction, it gives an automorphic cuspidal representation of ${\rm GL}_2(\mathbb{A}_\Q)$. The relation between the automorphic representations of ${\rm GL}_2(\mathbb{A}_\Q)$ and the cuspidal  cohomology of ${\rm GL}_2$  can be described in terms of the sheaf $\widetilde{ \mathcal M}_\mu$ associated to the highest weight representation $\mathcal M_\mu$ of $\rm GL(2,\R)$. This generalises to other classical groups. The first author and A. Raghuram studied the classification problem for the non-vanishing of cuspidal  cohomology of classical groups in \cite{bhagwat}.\smallskip
	
       There are important questions about estimating the dimensions of cuspidal cohomology. In \cite{Calegari-Emerton}, Calegari and Emerton have established upper bounds on multiplicities of cohomological representations of semisimple  real Lie groups. In \cite{marshall}, Marshall has studied the growth of cohomology with respect to growing weight of automorphic representations of ${\rm GL}_2$ over a number field which is not totally real.\smallskip

		 One can ask a related but different question: how much of cuspidal cohomology of  ${\rm GL}_2$ is obtained by automorphic induction of gr\"o{\ss}encharacters.  In a recent work, C. Ambi  (see \cite{ambi}) has proved  results about estimates of the growth of the cuspidal cohomology of ${\rm GL}_2$ and ${\rm GL}_3$ in terms of symmetric square transfer from ${\rm GL}_2$ (both lower and upper bounds).\smallskip

	One can ask a similar question for the group ${\rm GL}_4$, i.e. how much of the cuspidal cohomology of ${\rm GL}_4$ is obtained from ${\rm GL}_2$ by symmetric cube transfer. Kim and Shahidi \cite{kim-shahidi} have proved that the symmetric cube of a cuspidal automorphic representation $\pi$ of ${\rm GL}_2(\mathbb{A}_\Q)$ gives an automorphic representation of ${\rm GL}_4(\mathbb{A}_\Q)$ which is further cuspidal if $\pi$  is not dihedral i.e., it is not obtained by automorphic induction from a gr\"o{\ss}encharacter. The following is the main result of the present paper which gives an estimate of number of cuspidal automorphic representations of ${\rm GL}_4(\mathbb{A}_\Q)$ obtained from ${\rm GL}_2(\mathbb{A}_\Q)$ by symmetric cube transfer corresponding to a specific level structure. \medskip
	
	\begin{theorem}\label{main-theorem}
		Let $k\geq 2$ be an even integer and $p\geq 3$ be a prime. Let $E_k(p^n)$ denote the set of cuspidal automorphic representations of ${\rm GL}_4(\mathbb{A}_\Q)$ corresponding to the level structure $K_f^4(p^n)$ which are obtained by symmetric cube transfer of cuspidal automorphic representations of ${\rm GL}_2(\mathbb{A}_\Q)$ of weight $\left( \dfrac{k}{2}-1, 1- \dfrac{k}{2}\right)$. Then,
		$$ |E_k(p^{n})| ~\gg_k~ p^{n-1} \hspace{1 cm} \text{as} \hspace{.2 cm} n \longrightarrow \infty$$
		where the implied constant depends on $k$. 
	\end{theorem}\medskip
	
	Since the symmetric cube transfer of a cohomological representation $\pi$ of ${\rm GL}_2(\mathbb{A}_\Q)$ is also cohomological, the above result gives an estimate of the contribution of symmetric cube of representations of ${\rm GL}_2$ to the cuspidal cohomology of ${\rm GL}_4$.\medskip

	We briefly describe the structure of the remaining paper. After setting up the notations in Section \ref{preliminaries}, we discuss a particular calculation of conductors of representations of $p$-adic groups and give an important inequality result in Section \ref{calculations of conductors} (see Proposition \ref{proposition-conductor-inequality}). In Section \ref{proof-of-main-theorem}, we give the proof of Theorem \ref{main-theorem} using the aforementioned work of Kim \& Shahidi \cite{kim-shahidi}, C. Ambi \cite{ambi}, Raghuram \cite{Raghuram} and other results developed in previous sections. In the end we discuss some open questions and our ongoing work in this direction. \bigskip
	
		{\it Acknowledgements:} 
Both the authors are thankful to A. Raghuram and Chaitanya Ambi for valuable discussions regarding this paper. Chandrasheel Bhagwat would like to acknowledge the support of MATRICS research grant MTR/2018/000102 from the Science and Engineering Research Board, Department of Science and Technology, Government of India during this work. Sudipa Mondal would like to acknowledge the support of CSIR PhD fellowship during this work.\bigskip

\section{Preliminaries}\label{preliminaries}

\subsection{Symmetric cube transfer} We describe the construction of the symmetric cube transfer for a cuspidal automorphic representation $\pi$ of ${\rm GL}_2(\mathbb{A_{\mathbb F}})$ for a number field $\mathbb F$ (though it is well known to the experts).\smallskip

Let $\pi= \bigotimes \limits_v \pi_v$ be a cuspidal automorphic representation of ${\rm GL}_2(\mathbb{A}_{\mathbb{F}})$. For each place $v$, let $\phi_v$ be the two dimensional representation of the Weil-Deligne group attached to $\pi_v$. Let ${\rm sym}^3$ denote the $4$-dimensional irreducible representation of ${\rm GL}_2(\mathbb C)$. Then ${\rm sym}^3 \circ \phi_v $ is a four dimensional representation of the Weil-Deligne group. Using the Local Langlands correspondence, this gives an irreducible representation of ${\rm GL}_4(\mathbb{F}_v)$, denoted by ${\rm sym}^3(\pi_v)$. \smallskip

Define the automorphic representation $\text{sym}^3(\pi)$ of ${\rm GL}_4(\mathbb{A_{\mathbb F}})$ by
$$\text{sym}^3(\pi)= \bigotimes \limits_v \text{sym}^3(\pi_v).$$ 
This representation $\text{sym}^3(\pi)$ is known as the symmetric cube transfer of $\pi$.   \smallskip

Note: In this paper, we only consider the case $\mathbb F = \mathbb Q$. \medskip

%Write $\text{sym}^3(\phi_v) \text{ as } \text{sym}^3 \circ \phi_v$. 

%Under a change of variable, we have,
%$$\text{sym}^3(g)=  \begin{bmatrix} a^3 & a^2b & ab^2 & b^3 \\ 3a^2c & 2abc+a^2d & 2abd+b^2c & 3b^2d \\ 3ac^2 & 2acd+bc^2 & 2bcd+ad^2 & 3bd^2 \\ c^3 & c^2d & cd^2 & d^3 \end{bmatrix} \in {\rm GL}_4(\mathbb{C}) \quad \forall~ g= \begin{bmatrix} a & b \\ c & d  \end{bmatrix} \in {\rm GL}_2(\mathbb{C})$$
			
	\subsection{Dimension formulae for spaces of cusp forms}
	
%	\Gamma_0(N) & = \left\{A \in %\text{SL}_2(\mathbb{Z}): A \equiv %\begin{bmatrix}
%	* & * \\
%	0 & * 
%	\end{bmatrix}(\text{mod } N) % \right\},\\
	
	Let $N$ be a positive integer. Consider the following congruence subgroup of $\text{SL}_2(\mathbb{Z})$,
	\begin{align*}
	\Gamma_1(N) & = \left\{A \in \text{SL}_2(\mathbb{Z}): A \equiv \begin{bmatrix}
	1 & * \\
	0 & 1 
	\end{bmatrix}(\text{mod } N) \right\} .
    \end{align*}
    
	For an integer $k$, let $S_k(\Gamma_1(N))$ and $S_k^\text{new}(\Gamma_1(N))$ are the spaces of cusp forms of weight $k$ and newforms of weight $k$, respectively for the congruence subgroup $\Gamma_1(N)$. The following well-known results will be important in the proof of the main theorem (see  \cite{miyake} for the details and \cite{ambi} for the formulation of these results as given below).\medskip
	
\noindent  Notation: 	Let $f$ and $g$ be two non-negative real valued functions on $\mathbb{N}$. We write $f \ll g$ (or equivalently $g \gg f$)  if there exists a constant $C$ and $n_0\in \mathbb{N}$ such that 
$$f(n) \leq Cg(n) \quad \forall ~ n\geq n_0.$$ 
	
	We write $f \sim g$ if both $f \ll g$ and $g \ll f$ hold.\medskip
	
	\begin{theorem}\label{dimension-estimate-cusp-forms}
		\emph{[A dimension formula for cusp forms for $\Gamma_1(N)$]}
		 $$\text{dim}_{\mathbb{C}}~S_k(\Gamma_1(N)) \sim_k N^2 \quad \text{ as }N \longrightarrow \infty. $$
		where the implied constant depends on $k$.
	\end{theorem} \smallskip
	
	\begin{theorem}\label{dimension-estimate-new-forms} \emph{[A dimension formula for newforms for $\Gamma_1(N)$]} 
	For $n \geq 1$ and a prime $p \geq 2$, we have
		
		$$  \frac{\text{dim}_\mathbb{C}~(S_k^\text{new}(\Gamma_1(p^n)))}{p^{2n}} = \left( \frac{k-1}{4\pi^2} \right) \left(1-\frac{1}{p^2}\right)^2 + o(1) \quad \text{ as } n \longrightarrow \infty.$$
	\end{theorem} \medskip

		Let $C_k(N)$ denote the set of normalised cusp eigenforms of Hecke operators of $\Gamma_1(N)$ of weight $k$ obtained by automorphic induction of gr\"o{\ss}encharacters of imaginary quadratic extensions. Define $ N' := \prod \limits_{p|N} p$, i.e. the product of all distinct prime factors of $N$. The following upper bound estimate was proved by C. Ambi (see \cite{ambi}).\medskip
		
		\begin{theorem}\label{chaitanya-upper-bound} 
	Let $k \geq 1, \epsilon \in (0,1)$. Then,
	$$ |C_k(N)| ~\ll_{k,\epsilon}~ N \cdot N'^{(1+\epsilon)} \text{ } \text{ } \quad \text{as} ~N \longrightarrow \infty. $$
	\end{theorem}
	\medskip

We restate the above theorem in the special case as below.

\begin{proposition}\label{special-case}
Let $p$ be an odd prime. Then for $k\geq 2,~ \epsilon \in (0,1)$, we have,
$$|C_k(p^n)| ~\ll_{k,\epsilon} ~p^{n+1+\epsilon} \text{ as } ~n \longrightarrow \infty$$

\end{proposition}\medskip

We use the following two lemmas in the proof of \ref{special-case}.

\begin{lemma}\label{p=3mod4}
		Let $p \equiv 3 \mod 4$ and $\mathbb{E}=\Q(\sqrt{-p})$. Suppose $f \subseteq \mathcal{O}_{\mathbb{E}}$ such that  $N_{\mathbb{E}/\mathbb{F}}(f) \cdot \allowbreak |D_{\mathbb{E}/\mathbb{Q}}|| p^n$ for some $n\geq 1$, then $f |p^m \mathcal{O}_{\mathbb{E}}$ where $m= \frac{n}{2}$ if $n$ is even or $m= \frac{n-1}{2}$ if $n$ is odd.
\end{lemma}
\medskip

\begin{proof}
Since $p \equiv 3 \mod 4 \text{ that is } -p \equiv 1 \mod 4, \text{ we have } D_{\mathbb{E}/\mathbb{Q}}=-p$. Note that $p$ ramifies in $\mathbb{E}$ as $p | |D_{\mathbb{E}/\mathbb{Q}}|$ and there is only one prime above $p$, say $\mathfrak{p}$ with $N_{\mathbb{E}/\mathbb{F}}(\mathfrak{p})=p$. Since $N_{\mathbb{E}/\mathbb{F}}(f) | p^{n-1}$, we have $f=\mathfrak{p}^t$ for some $t \leq n-1$ and $p^t | p^{n-1} $. Since $p\mathcal{O}_{\mathbb{E}}=\mathfrak{p}^2,$ we have
$$ \mathfrak{p}^{n-1}|p^m\mathcal{O}_{\mathbb{E}} \implies f=\mathfrak{p}^t|p^m\mathcal{O}_{\mathbb{E}} .$$
\end{proof}\smallskip

\begin{lemma}\label{p=1mod4}
Let $p \equiv 1 \mod 4$ be a prime. Then $\forall \text{ } k\geq 2 \text{ and } \forall \text{ } n\geq1$, we have, $|C_k(p^n)|=0.$
\end{lemma}\smallskip

The proof of Lemma \ref{p=1mod4} follows from the proof of \cite[Lemma 3.3]{ambi}.\medskip

We sketch the outline of the proof of Proposition \ref{special-case} here. We need to only consider the case $p \equiv 3 \mod 4$ in view of  Lemma \ref{p=1mod4}.\medskip

%\begin{theorem}[Ambi]
%Let $p$ ba a odd prime. Then for $k\geq 2, \epsilon %\in (0,1)$, we have,
%$$|C_k(p^n)| ~\ll_{k,\epsilon} ~p^{n+1+\epsilon} %\text{ as } ~n \longrightarrow \infty$$
%\end{theorem}

\begin{proof}
%We sketch the outline of the proof here. We need to only consider the case $p \equiv 3 \mod 4$ in view of  Lemma \ref{p=1mod4}. \\
Let, for an ideal $ \mathfrak{j} \subset \mathcal{O}_{\mathbb{E}}$, $h_{\mathbb{E}}^0(\mathfrak{j})$  and $h_{\mathbb{E}}'(\mathfrak{j})$ denote the number of Hecke characters with conductor exactly $\mathfrak{j}$ and the number of characters of the narrow ray class group modulo $\mathfrak{j}$ respectively. Hence $\sum \limits_{f|\mathfrak{j}} h_{\mathbb{E}}^0(f)=h_{\mathbb{E}}'(\mathfrak{j})$. we have,
$$|C_k(p^n)| \leq \sum \limits_{\substack{\mathbb{E}=\mathbb{Q}(\sqrt{-p}), \\  N_{\mathbb{E}/\mathbb{Q}}(f) \cdot|D_{\mathbb{E}/\mathbb{Q}}||p^n}} h_{\mathbb{E}}^0(f) .$$

By Lemma \ref{p=3mod4}, 
$$|C_k(p^n)|\leq \sum \limits_{\substack{f|p^m\mathcal{O}_{\mathbb{E}} \\ \mathbb{E}=\mathbb{Q}(\sqrt{-p}) }} h_{\mathbb{E}}^0(f)= h'_{\mathbb{E}}(p^m\mathcal{O}_{\mathbb{E}}) \leq  h_{\mathbb{E}}\Phi_{\mathbb{E}}(p^m\mathcal{O}_{\mathbb{E}}) $$ 
where $h_{\mathbb{E}}=$ the class number of $\mathbb{E}$ and $\Phi_{\mathbb{E}}(\mathcal{I})= N_{\mathbb{E}/\mathbb{Q}}(\mathcal{I})\prod \limits_{\substack{\mathfrak{p}|\mathcal{I}, \\ 
\mathfrak{p} \text{ prime }}} (1-N_{\mathbb{E}/\mathbb{Q}}(\mathfrak{p})^{-1})$ with $\Phi_{\mathbb{E}}(\mathcal{O}_{\mathbb{E}})=1$. The last inequality follows from \cite[Theorem 3.25(i)]{narkiewicz}. The rest follows similarly as in \cite[Theorem 1.1]{ambi}. 

\end{proof}\medskip

	\subsection{The Level Structure :}  \label{level-structure}
	Consider the algebraic group ${\rm GL}_m$.  For a finite prime $p$ of $\Q$, let $\Z_p$ be the ring of $p$-adic integers. For each integer $n \geq 0$, define,
	$$ K_p^m(n):= \{ x=(x_{i,j})_{m\times m} \in {\rm GL}_m(\mathbb{Z}_p) : x_{m,k}\in p^n\mathbb{Z}_p \text{ } \text{for} \text{ } 1\leq k<m, \text{ } x_{m,m}-1 \in p^n\mathbb{Z}_p\}.$$
	
	Let $\mathbb{A}_f$ denote the finite part of ad{\`e}les over $\Q$. Let $N$ be a positive integer with prime factorisation $N =  \prod \limits_{i=1}^{r}p_i^{n_i}$ . Define a compact open subgroup $K_f^m(N) = \prod \limits_p K_p$ of ${\rm GL}_m(\mathbb{A}_f)$ where 
	$$K_p =
	\begin{cases} 
	& K_{p_i}^m{(n_i)} \quad \text{if} ~ p \mid N~\text{i.e.},~ p = p_i  \\
	& {\rm GL}_m(\mathbb{Z}_p) \quad \text{if}~ p \nmid N
	\end{cases}
	$$
	
	For each $N \ge 1$, the compact open subgroup $K_f^m(N) \subseteq {\rm GL}_m (\mathbb A_f)$ is called the level structure corresponding to $N$. \medskip
	
		For an even integer $k \geq 0$, let $D_k(N)$ be the set of all cuspidal automorphic representations of 
${\rm GL}_3(\mathbb{A}_\Q)$ corresponding to the level structure $K_f^3(N)$ used by Ambi in \cite{ambi} (which is slightly different than the level structure we are using) and highest weight $\mu_k :=(k-2,0,2-k)$
which are obtained by symmetric square transfer of
 cuspidal automorphic representations of ${\rm GL}_2(\mathbb{A}_\Q)$ 
 of every possible level structure. The following result was proved by C. Ambi in \cite{ambi} which gives a lower bound on $|D_k(p^n)|$.
	
\begin{theorem}\label{chaitanya-lower-bound}  
	For $k \geq 2 \text{ and } p>2$, we have
$$ |D_k(p^n)| \gg_k p^{2n/5} \hspace{1 cm} as \quad n \longrightarrow \infty, $$ 
	where the implied constant depends on $k$.
\end{theorem}

\subsection{A generalisation to ${\rm GL}_4$}
Let $E_k(N)$ denote the set of cuspidal automorphic representations of ${\rm GL}_4(\mathbb A_\Q)$ corresponding to the level structure $K_f^4(N)$ and highest weight $\nu_k$ which are obtained by symmetric cube transfer of cuspidal automorphic representation of ${\rm GL}_2(\mathbb{A}_\Q)$ of weight $\lambda_k = \left( \dfrac{k}{2}-1, 1- \dfrac{k}{2}\right)$. See Lemma \ref{cohomological-sym-cube} and the preceding discussion for the description of the weight $\nu_k$.\medskip

\noindent Note that $E_k(N)$ is analogous to the set $D_k(N)$ defined earlier. The main result of this paper, Theorem \ref{main-theorem} (stated below again) gives an asymptotic estimate of the size of the set $E_k(N)$ generalising Theorem \ref{chaitanya-lower-bound}.

\begin{theorem*}
	Let $k\geq 2$ be an even integer and $p\geq 3$ be a prime. Then
	$$ |E_k(p^{n})| ~\gg_k~ p^{n-1} \hspace{1 cm} \text{as} \hspace{.2 cm} n \longrightarrow \infty$$
	where the implied constant depends on $k$. 
\end{theorem*}

\section{Some calculations of conductors}\label{calculations of conductors}
 
\subsection{Conductor of a representation :}

Let $(\rho,\mathcal H)=(\bigotimes \limits_{p\leq \infty}\rho_p, \bigotimes \limits_{p\leq \infty}\mathcal H_p)$ be an irreducible automorphic representation of ${\rm GL}_m(\mathbb{A}_\Q)$. 

\begin{defn} For each finite prime $p$, the conductor of $\rho_p$ is defined to be  the smallest  integer $c(\rho_p)\geq 0$ such that the set $\mathcal H_p^{K_p^m({{c(\rho_p)}})}$ consisting of all $K_p^m({{c(\rho_p)}})$-fixed vectors of $\mathcal H_p$ is non-zero. The conductor of $\rho$ is defined as 
	$$N_\rho= \prod_{p<\infty} p^{c(\rho_p)} .$$ 
\end{defn}
	
	\subsection{Conductor of a local character:} Let $\mathcal{K}$ be a non-archimedean local field of characteristic zero. Let $\mathcal{O}_{\mathcal{K}}$, $\mathfrak{p}$, $\kappa$ be its ring of integers, the maximal ideal and the residue field respectively. 
		
	\begin{defn}
	The conductor $c(\zeta)$ of a character $\zeta$ of $\mathcal{K}^\times$ is the smallest positive integer $n$ such that ${\zeta}_{|_{1+\mathfrak{p}^n}} = 1 $. 
	\end{defn}
	
	Note that, for $i\geq0$, $c(\zeta^i)\leq c(\zeta)$.\smallskip

	\subsection{Conductor of symmetric cube of a representation:} 
	Suppose the residue field $\kappa$ is of characteristic $p > 0$ and let $q = | \kappa |$. 

	Let $W_\mathcal{K}$ be the Weil group of $\mathcal{K}$. Consider a finite dimensional vector space $V$ and a continuous homomorphism i.e. a representation of $W_\mathcal{K}$ given by $$\phi : W_\mathcal{K} \rightarrow {\rm GL}(V).$$

	Let $\omega$ be a character of $W_\mathcal{K}$ defined by
	\[g: x \mapsto x^{\omega(g)} \quad \forall~ x \in \overline{\kappa}.
	\]
	If there is a nilpotent endomorphism $N$ of $V$ which satisfies
	$$\phi(g)N\phi(g)^{-1}= \omega(g)N \quad \forall ~g \in W_\mathcal{K},$$ then the pair $(\phi,N)$ is said to be Weil-Deligne representation. See \cite{tate} for more details.\medskip
	% Here, Sudipa will add the details of how the character $\omega$ is defined.
	
	By the local Langlands correspondence, a two-dimensional Weil-Deligne representation $(\phi,N)$ corresponds uniquely to an irreducible admissible representation $\pi$ of ${\rm GL}_2(\mathcal{K})$ and vice-versa. 
	We refer to $(\phi,N)$ as the local parameter of $\pi$. For a supercuspidal representation $\pi$ of ${\rm GL}_2(\mathcal{K})$, the local parameter of $\pi$ is an irreducible representation of $W_\mathcal{K}$ with $N=0$.\medskip
	
	For $p \geq 3$, every supercuspidal representation $\pi$ of ${\rm GL}_2(\mathcal{K})$ is a dihedral supercuspidal representation (see \cite{bump}) i.e., there exists a quadratic extension $\mathcal{F}$ over $\mathcal{K}$ and a character $\eta$ of $W_\mathcal{F}$	such that $\pi$ has a local parameter $\phi$ of the form
	$\phi= {\rm Ind}^{W_\mathcal{K}}_{W_\mathcal{F}} ~ \eta$
	with  $\eta \neq \eta^{\sigma},$  where $\sigma$ is a non-trivial element in $W_\mathcal{K} \setminus W_\mathcal{F}$~
	 (For $x\in W_\mathcal{F},~ \eta^{\sigma}$ is defined as $\eta^{\sigma}(x)=\eta(\sigma x \sigma^{-1})$). Note that $W_\mathcal F$ is an index-$2$ subgroup of $W_\mathcal K$. With respect to a suitable basis, $\phi$ has the following matrix form. $$  \phi(x)=\begin{bmatrix} \eta(x) & \\ & \eta^{\sigma}(x) \end{bmatrix},~ x \in W_\mathcal{F} \text{ and } \phi(\sigma)=\begin{bmatrix} & 1 \\ \eta(\sigma^2) & \end{bmatrix}. $$
	
	Let $\zeta$ be a character of $\mathcal{F}^\times$ corresponding to the character $\eta$ of $W_\mathcal{F}$ via the Artin isomorphism. We denote the dihedral supercuspidal representation $\pi$ by $w_{\mathcal{F},\zeta}$. If the conductor of $\zeta$ is $c(\zeta)$, then the conductor $c(\phi)$ of $\phi$ is given by
	$$ c(\phi)= {\rm dim}({\zeta})~v(d_{\mathcal{F}/\mathcal{K}}) + f_{\mathcal{F}/\mathcal{K}}~ c(\zeta),$$
	where $d_{\mathcal{F}/\mathcal{K}}$ is the discriminant of the field extension $\mathcal{F}/\mathcal{K}$ and $f_{\mathcal{F}/\mathcal{K}}$ is the residue class degree (see \cite{serre}, \cite{manami} for more details).\medskip

	Since $\zeta$ is a character, dim$(\zeta)=1$. If the characteristic of $\kappa$ is odd, we have
	\[
	c(\phi) = \begin{cases}
	2c(\zeta) & \text{ if } \mathcal{F} / \mathcal{K} \text{ is unramified} \\
	1+c(\zeta) & \text{ if } \mathcal{F} / \mathcal{K} \text{ is ramified}
	\end{cases} \]
	
	The central character of $w_{\mathcal{F},\zeta}$ of
	 ${\rm GL}_2(\mathcal{K})$ is $\zeta|_{\mathcal{K}^\times} \cdot \chi_{\mathcal{F}/\mathcal{K}}$
	  where $\chi_{\mathcal{F}/\mathcal{K}}$ is the quadratic character of $\mathcal{K}^\times$ 
	  associated to the quadratic extension $\mathcal{F}/\mathcal{K}$ 
	  such that ${\chi_{\mathcal{F}/\mathcal{K}}}_{|_{N_{\mathcal{F}/\mathcal{K}}(\mathcal{F}^{\times})}} = 1.$ \medskip 
	  
	 Define a character $\zeta^\sigma$ of $\mathcal{F}^\times$ by $\zeta^\sigma(x)=\zeta(\sigma(x))$. By Artin isomorphism, the conjugate character $\eta^\sigma$ of $W_{\mathcal{F}} $ corresponds to $\zeta^\sigma$. If the central character of $w_{\mathcal{F},\zeta}$ is trivial, then $\zeta^\sigma=\zeta^{-1}$, equivalently $\eta^\sigma=\eta^{-1}$.\medskip
	  
	Since $\text{sym}^3: \text{GL}_2(\mathbb{C}) \to \text{GL}_4(\mathbb{C})$ is defined as 
	$$\text{sym}^3 \left(\begin{bmatrix} a & b \\ c & d \end{bmatrix} \right) =  \begin{bmatrix} a^3 & a^2b & ab^2 & b^3 \\ 3a^2c & 2abc+a^2d & 2abd+b^2c & 3b^2d \\ 3ac^2 & 2acd+bc^2 & 2bcd+ad^2 & 3bd^2 \\ c^3 & c^2d & cd^2 & d^3 \end{bmatrix},$$ 
	we have,
	$$ \text{sym}^3(\phi)(x) = \begin{bmatrix} \eta^3(x) & & & \\ & \eta^2(x)\eta^\sigma(x) & & \\ & & \eta(x)(\eta^\sigma)^2(x) & \\ & & &(\eta^\sigma)^3(x) \end{bmatrix} \text{ } \text{ and}$$
	$$ \text{sym}^3(\phi)(\sigma) = \begin{bmatrix} & & & 1 \\ & & \eta^\sigma(\sigma^2) & \\ & (\eta^\sigma)^2(\sigma^2) & & \\ (\eta^\sigma)^3(\sigma^2) & & & \end{bmatrix} \text{ since }\eta^\sigma(\sigma^2)= \eta(\sigma^2).$$\medskip
	
	Hence $\text{sym}^3(\phi)$ is a sum of two $2$-dimensional representations i.e.,
	 $$\text{sym}^3(\phi)= {\rm Ind}^{W_\mathcal{K}}_{W_\mathcal{F}}(\eta^3) \oplus {\rm Ind}^{W_\mathcal{K}}_{W_\mathcal{F}} (\eta^2\eta^\sigma)$$
	
	Since we are considering supercuspidal representations with trivial central character, we have $\eta^\sigma =\eta^{-1}$ which gives,
	$$\text{sym}^3(\phi)= {\rm Ind}^{W_\mathcal{K}}_{W_\mathcal{F}}(\eta^3) \oplus {\rm Ind}^{W_\mathcal{K}}_{W_\mathcal{F}} (\eta).$$
	
	Further assume $\eta^6\neq1$. This gives $\eta^3 \neq (\eta^3)^\sigma$ that is, ${\rm Ind}^{W_\mathcal{K}}_{W_\mathcal{F}}(\eta^3)$ is irreducible. Hence with all these assumptions, we have,
	$$ c \left( \text{sym}^3(\phi) \right)= c \left({\rm Ind}^{W_\mathcal{K}}_{W_\mathcal{F}}(\eta^3) \right) + c \left({\rm Ind}^{W_\mathcal{K}}_{W_\mathcal{F}} (\eta)\right).$$\medskip
	Since $c \left({\rm Ind}^{W_\mathcal{K}}_{W_\mathcal{F}}(\eta^3) \right)= v(d_{\mathcal{F}/\mathcal{K}}) + f_{\mathcal{F}/\mathcal{K}}~ c(\zeta^3)$, we have the following lemma.
	
	\begin{lemma}\label{lemma-conductor-sym-cube}
	Let $\mathcal{K}$ be a non-archimedean local field with the characteristic of $ \kappa $ odd prime. 
	Suppose $\pi=w_{\mathcal{F},\zeta}$ is a supercuspidal representation of ${\rm GL}_2(\mathcal{K})$ with the trivial central character such that $\zeta^6\neq 1$ where $\mathcal{F} \text{ and } \zeta$ are as described above. Then the formula for the conductor of $\rm{sym}^3(\pi)$ can be described as follows.
	\[
	c(\rm{sym}^3\pi) = \begin{cases}
	2c(\zeta^3)+2c(\zeta) & \text{ if } \mathcal{F} / \mathcal{K} \text{ is unramified}, \\
	c(\zeta^3)+c(\zeta)+2 & \text{ if } \mathcal{F} / \mathcal{K} \text{ is ramified}.
	\end{cases} \]
	\end{lemma}
	
	Consider $\mathcal{K}= \Q_p$. Then the relation between $c(\pi_p)$ and $c(\text{sym}^3\pi_p)$ is captured by the following proposition.\smallskip
	 	
	\begin{proposition}\label{proposition-conductor-inequality}
		Let $p \geq 3$ be a finite prime. Suppose $\pi_p=w_{\mathcal{F},\zeta}$ is a dihedral supercuspidal representation of $ \rm{GL}_2(\Q_p)$ \rm{(as above)}. Further assume that the central character of $w_{\mathcal{F},\zeta}$ is trivial and $\zeta^6\neq 1$. Then $c(\pi_p) \leq c(\rm{sym}^3\pi_p) \leq 2c(\pi_p)$.
	\end{proposition}
	
	\begin{proof}
	%	Since $p\geq 3$, every supercuspidal representation is a dihedral supercuspidal representation. Hence $\pi_p=\omega_{\mathcal{F},\zeta}$ for some quadratic field extension $\mathcal{F}/\Q_p$ and some character $\zeta$ of $\mathcal{F}^\times$ which is nontrivial on the kernel of the norm map $N_{\mathcal{F}/\Q_p}$ of $\mathcal{F}^\times$. Further assume that the central character of $\omega_{\mathcal{F},\zeta}$ is trivial.\medskip
		
	 We use Lemma \ref{lemma-conductor-sym-cube}.	If $\mathcal{F}/\Q_p$ is unramified, then 
		\[ 
		\begin{split} c(\text{sym}^3\pi_p) & = 2c(\zeta^3)+2c(\zeta), \\
		c(\pi_p)& =2c(\zeta).
		\end{split}
		\]
		
		So we have $c(\pi_p) \leq c(\text{sym}^3\pi_p).$ \medskip
		
		Since $2c(\zeta^3)\leq 2c(\zeta)$, we conclude that
		$$c(\text{sym}^3\pi_p)\leq 4c(\zeta)=2c(\pi_p).$$ 
		
		On the other hand if $\mathcal{F}/\Q_p$ is ramified, then 
		\[ 
		\begin{split}
		c(\text{sym}^3\pi_p) & =c(\zeta^3)+c(\zeta)+2, \\ 
		c(\pi_p) & = c(\zeta)+1.
		\end{split}
		\]
		 
So we have $c(\pi_p) \leq c(\text{sym}^3\pi_p).$
Since $c(\zeta^3)\leq c(\zeta)$, it follows that $c(\text{sym}^3\pi_p)\leq 2c(\pi_p)$. \medskip
		 
		 We conclude that the desired inequality $c(\pi_p) \leq c(\text{sym}^3\pi_p) \leq 2 c(\pi_p)$ holds in all cases.    \medskip
\end{proof}

\section{Proof of Theorem \ref{main-theorem}}\label{proof-of-main-theorem}
	
	\subsection{The symmetric cube transfer from ${\rm GL}_2$  and cohomology of ${\rm GL}_4$}
	
	Let $\rho$ be a cuspidal automorphic representation of ${\rm GL}_m/\Q$ with $\rho_\infty$ and $\rho_f$ be its finite and infinite parts respectively. Let $\mu$ be a dominant integral weight corresponding to the standard Borel subgroup and $\mathcal M_{\mu,\mathbb{C}}$ be the underlying vector space of the finite dimensional irreducible highest weight representation corresponding to $\mu$. For a given level structure $K_f^m \subset  {\rm GL}_m(\mathbb{A}_f)$, we denote by ${\rho}^{K_f^m}_f$, the space of $K_f^m$-fixed vectors of $\rho_f$. We write $\rho \in \text{Coh}({\rm GL}_m, \mu, K_f^m)$ if  the relative Lie algebra cohomology with coefficients in $\mathcal M_{\mu,\mathbb{C}}$ is non zero in some degree i.e.,
	$$ H^*(\mathfrak{gl}_m, \mathbb{R}_+^\times \cdot {\rm SO}_m(\mathbb{R}),\rho_\infty \bigotimes \mathcal M_{\mu,\mathbb{C}})\bigotimes \rho_f^{K_f^m} \neq 0 .$$
	
	For more details, the reader is referred to \cite{bhagwat}.\medskip
	
	The following theorem by Kim and Shahidi \cite{kim-shahidi} gives the sufficient condition for the symmetric cube of an automorphic representation of ${\rm GL}_2$ being cuspidal.
	
\begin{theorem}\label{kim-shahidi-theorem}
Let $\mathbb{F}$ be a number field and $\pi$ be a automorphic cuspidal representation of $\mathrm{GL}_2(\mathbb{A}_\mathbb{F})$. Then $\mathrm{sym}^3(\pi)$ is automorphic representation of $\mathrm{GL}_4(\mathbb{A}_\mathbb{F})$. Further, $\mathrm{sym}^3(\pi)$ is cuspidal unless $\pi$ is dihedral or tetrahedral type. In particular, if $\mathbb{F}=\Q$, and $\pi$ is the automorphic cuspidal representation attached to a non-dihedral holomorphic form of weight $\geq2$, then $\mathrm{sym}^3(\pi)$ is cuspidal. 
\end{theorem}
	
	A cusp form $\pi$ of ${\rm GL}_2(\mathbb{A}_\mathbb{F})$ is called dihedral iff $\pi \cong \pi \bigotimes \xi$ for some character $\xi$ such that $\xi^2=1$. In this case there is a gr\"o{\ss}encharacter $\chi$ of the quadratic extension $\mathbb{E}/\mathbb{F}$ associated to $\xi$ such that $\pi$ is obtained by automorphic induction of $\chi$. Hence we conclude from Theorem \ref{kim-shahidi-theorem} that if the automorphic cuspidal representation $\pi$ of ${\rm GL}_2(\mathbb{A}_\Q)$ is not obtained by automorphic induction from a gr\"o{\ss}encharacter, then $\text{sym}^3{\pi}$ is cuspidal.\medskip
	%Note that, $E_k(N)$ is analogous to the set $D_k(N)$. Now, we shall state our main theorem.
	%$$ E_k(N) = \{ \Pi\in A_{cusp}(GL_4/\Q,\nu_k,K_f^4(N)) : \exists \pi \text{ } \text{and} \text{ } M\geq1 \text{ } \text{ such that} \text{ } $$ $$ \pi \in A_{cusp}(GL2/\Q, \lambda_k,K_f^2(M)), \Pi=sym^3(\pi) \} $$

	Now we discuss the relation between the symmetric cube transfer and cuspidal cohomology of ${\rm GL}_4$.
		Let $p\geq 3$ be a prime and $\pi=\bigotimes \limits_q \pi_q$ be a cohomological cuspidal automorphic representation of $\mathrm{GL}_2(\mathbb{A}_\Q)$ such that $\pi_p$ is supercuspidal. Assume that the highest weight corresponding to $\pi_\infty$ is $\lambda_k:=\left(\dfrac{k}{2}-1,~ 1-\dfrac{k}{2}\right)$ where $k\geq2$.\smallskip
		
		Let $\Pi=\mathrm{sym}^3(\pi)$ be the representation of $\mathrm{GL}_4(\mathbb{A}_\Q)$ obtained by symmetric cube transfer. Let $c(\Pi)$ be its conductor and $\nu_k = \left( 3 \left(\dfrac{k}{2}-1\right),~ \dfrac{k}{2}-1,~ 1 - \dfrac{k}{2},~   3 \left( 1  - \dfrac{k}{2} \right) \right)$ be the highest weight corresponding to $\Pi_\infty$. We conclude from Raghuram's work 
		(see Theorem 3.2, \cite{Raghuram}), that $\Pi$ is cohomological w.r.t. the weight $\nu_k$. We summarise this discussion as a lemma here. \medskip
		
		%This  (possibly after tensoring with a power of the determinant):
		
		\begin{lemma}\label{cohomological-sym-cube}
		For an odd prime $p$ and a cohomological cuspidal automorphic representation $\pi$  of $\mathrm{GL}_2(\mathbb{A}_\Q)$ as above, we have
		\begin{align*}
		&\Pi \in {\mathrm{Coh}(\mathrm{Gl}_4/\Q,\nu_k,K_f^4(c(\Pi))}, \\  
		 \text{i.e.}~&H^{\bullet}(\mathfrak{gl}_4, \mathbb{R}_+^\times \cdot {\rm SO}_4(\mathbb{R}); \rho_\infty \bigotimes M_{\nu_{k},\mathbb{C}})\bigotimes \Pi_f^{K_f^4(c(\Pi))} \neq 0 .
		\end{align*}
	\end{lemma} \medskip
	
	\subsection{Proof of Theorem \ref{main-theorem} :}
	Now we give a proof of Theorem \ref{main-theorem} using all the results mentioned before.
	
	\begin{proof}
		Let $p \geq 3$ be a prime and  $\pi=\bigotimes \limits_{q\leq \infty}\pi_q$ be an irreducible cuspidal automorphic representation of ${\rm GL}_2(\mathbb{A}_\Q)$ corresponding to the subgroup $K_f^2(p^n), n\geq1$. From Theorem ~\ref{kim-shahidi-theorem}, we conclude that either $\text{sym}^3(\pi)\in \bigcup \limits_{j \geq 1}E_k(p^j)$ or $\pi$ is obtained by automorphic induction from a gr\"o{\ss}encharacter. Also we have, $\pi_p$ corresponds to a unique newform in $S_k^\text{new}(\Gamma_1(p^i))$ for some $1\leq i\leq n$ with some character. \smallskip
		%We will count the contribution of some specific supercuspidal representations $\pi_p$ to $\bigcup \limits _{j\geq1}E_k(p^j)$.\medskip
		In order to get a lower bound for $|E_k(N)|$, it suffices to consider only those $\pi$ for which, $\pi_p$ is supercuspidal representation  and $sym^3(\pi)$ contributes to $\bigcup \limits _{j\geq1}E_k(p^j)$.  \smallskip
		%We do not  consider the cases where $\pi_p$ is a principal series or Steinberg representation. 

		If $c(\pi_p)$ is the conductor of the supercuspidal representation $\pi_p$ and it corresponds to a newform in $S_k^\text{new}(\Gamma_1(p^i))$, then $c(\pi_p)=i$. Using Proposition \ref{proposition-conductor-inequality}, we conclude that $\pi_p$ can contribute to $\bigcup \limits_{j\geq1}E_k(p^j)$ only if 
		$c(\pi_p)\leq j \leq 2c(\pi_p).$\medskip
	
	Since $1\leq c(\pi_p) \leq n \implies 1\leq j \leq 2n $, we conclude that every supercuspidal representation $w_{\mathcal{F},\zeta}$ such that $\zeta^6\neq 1$ within $\bigoplus \limits_{1\leq i\leq n}S_k^\text{new}(\Gamma_1(p^i)) \backslash C_k(p^{n})$ with trivial central character is in the set $E_k(p^{2n}) \backslash E_k(p) $. \medskip

		For a prime $p\geq 3$, we have $$\frac{\text{dim}_\mathbb{C}~ S_k^\text{new}(\Gamma_1(p^i))}{p^{2i}} = \left(\dfrac{k-1}{4\pi^2} \right) \left(1-\dfrac{1}{p^2}\right)^2 + o(1) \quad \text{as}~ i~ \longrightarrow \infty.$$\smallskip
	
	Observe that  $1-\dfrac{1}{p^2} > \dfrac{3}{4}$ as $p\geq 3$. Hence, 
		$$ \sum_{1 \leq i \leq n} {\rm dim}_\mathbb{C} S_k^\text{new}(\Gamma_1(p^i))~ \gg_k ~p^{2n}.$$\medskip
		
		We conclude from Theorem \ref{chaitanya-upper-bound} and Proposition  \ref{special-case} that $|C_k(p^n)| \ll_k p^{n+2}$. Since $|C_k(p^n)|$ is negligible compared to $p^{2n}$, it gives,
		 $$|E_k(p^{2n})| ~\gg_k~ p^{2n} \text{ }\text{ as } n\longrightarrow \infty .$$ \smallskip
		  	 
	Since $K_f^m(p^{n-1}) \supseteq K_f^m(p^n)$, we have, $ E_k(p^{n-1}) \subseteq E_k(p^n)$. Hence when $n$ is odd i.e. $n-1$ is even, $$|E_k(p^{n})| \geq |E_k(p^{n-1})| ~\gg_k~ p^{n-1}~ \text{ as } n \longrightarrow \infty.$$ \smallskip
	
	Combining the two cases, we have 
	 $$|E_k(p^n)| ~\gg_k~ p^{n-1}~ \text{ as } n \longrightarrow \infty .$$	 
		 
	\end{proof} 
	
	From Theorem  \ref{dimension-estimate-cusp-forms}, we get the following estimate on the upper bound of $|E_k(p^{n})|$. \smallskip

	\begin{corollary}
		For a prime $p \geq 3$,
	$$|E_k(p^{n})| ~\ll_k~ p^{2n} \text{ } \text{ as } n \longrightarrow \infty .$$
	\end{corollary}
	\smallskip
	
	%This shows that $E_k(p^n) \sim p^{2n}$, that is the cuspidal automorphic representations of $GL_3(A_{\Q})$ obtained by symmetric cube transfer is strictly positive fraction of total cuspidal cohomology.
	
	\subsection{Further generalisations} As a part of ongoing work, we aim to answer the analogous question for the case of a number field $\mathbb{F}$, i.e. how much of cuspidal cohomology is obtained for ${\rm GL}_3$ and ${\rm GL}_4$ from ${\rm GL}_2$ by symmetric square transfer and symmetric cube transfer, respectively. 
		For $n \geq 5$, the question about estimating the cuspidal cohomology of ${\rm GL}_n$ in terms of symmetric power transfer from a lower rank group is interesting but seems to be more complicated as the results of Kim and Shahidi do not directly generalise to ${\rm GL}_n$ for $n \geq 5$.	
	
\end{document}